\documentclass{article}

\usepackage{blindtext} 

\usepackage[sc]{mathpazo} 
\usepackage[T1]{fontenc} 
\usepackage{microtype} 
\usepackage{bm} 
\usepackage[english]{babel} 

\usepackage[hmarginratio=1:1,top=32mm,columnsep=20pt]{geometry} 
\usepackage[hang, small,labelfont=bf,up,textfont=it,up]{caption} 
\usepackage{booktabs} 

\usepackage{lettrine} 

\usepackage{enumitem} 
\setlist[itemize]{noitemsep} 

\usepackage{abstract} 

\usepackage{titlesec} 
\renewcommand\thesection{\Roman{section}} 
\renewcommand\thesubsection{\arabic{subsection}} 
\titleformat{\section}[block]{\large\scshape\centering}{\thesection.}{1em}{} 
\titleformat{\subsection}[block]{\large}{\thesubsection.}{1em}{} 

\usepackage{fancyhdr} 
\usepackage{titling} 
\usepackage{graphicx}

\usepackage{hyperref} 
\usepackage[normal,tight,center]{subfigure}
\usepackage{xcolor}
\usepackage{float}

\usepackage{amsmath}
\usepackage{amsthm}
\usepackage{bbm}
\usepackage{amssymb}
\usepackage{amsfonts}
\usepackage{graphics}
\usepackage{caption}

\newtheorem{thm}{Theorem}

\newtheorem{lem}{Lemma}

\usepackage{graphicx}
\usepackage{ragged2e}

\pretitle{\begin{center}\Huge\bfseries} 
\posttitle{\end{center}} 
\title{Asymptotic Theory of Expectile Neural Networks} 
\author{%
\textsc{Jinghang Lin$^{1}$, Xiaoxi Shen$^{2}$, Qing Lu$^{2}$}\\
\normalsize 1.Department of Statistics and Probability, Michigan State University, East Lansing, Michigan 48823, U.S.A.\\
\normalsize 2.Department of Biostatistics, University of Florida, Gainesville, Florida 32611, U.S.A.\\
\normalsize \href{mailto:linjingh@msu.edu}{linjingh@msu.edu,} \href{mailto:lucienq@ufl.edu}{lucienq@ufl.edu}
}
\date{} 
\renewcommand{\maketitlehookd}{%
\begin{abstract}
Neural networks are becoming an increasingly important tool in applications. However, neural networks are not widely used in statistical genetics. In this paper, we propose a new neural networks method called expectile neural networks. When the size of parameter is too large, the standard maximum likelihood procedures may not work. We use sieve method to constrain parameter space. And we prove its consistency and normality under nonparametric regression framework.

\textbf{keywords: neural networks, method of sieves, consistency, normality}
\end{abstract}
}

\begin{document}

\maketitle


\section{Introduction}
Neural networks has been widely used in applications. However, the theoretical part of neural networks is not widely studied, especially statistical inference. From universal approximation theorem, a neural network with one hidden layer can approximate any continuous functions\cite{KMH}. For artificial neural networks, we use squared loss function. A unified treatment for the asymptotic normality of squared loss function could be find\cite{OMSP}.  In this paper, we use asymmetric squared loss function, which gives us a comprehensive view of conditional distribution and computation advantage. We focus on deriving the asymptotic a neural network with one hidden layer:
$$
y_i = \alpha_0 + \sum_{j = 1}^{r} \alpha_{j} \sigma(\gamma_{j}^{T}\mathbf{x}_{i} + \gamma_{0,j})
$$
We consider neural networks from statistical perspective. We rewrite the neural network in the regression form and make some assumptions:
$$
y_{i} = f_0(\mathbf{x_i}) + \epsilon_i,
$$
where $\epsilon_1, ..., \epsilon_n$ are $i.i.d$ random variables defined on $(\Omega, \mathcal{A}, \mathbb{P}) $with $E(\epsilon) = 0$ and $E(\epsilon^2) = \sigma^2 < \infty$. $f_0 \in \mathcal{F}$ is an unknown function, where $\mathcal{F}$ is the class of continuous function. However, if the complexity of $\mathcal{F}$ is the large, the estimator may be inconsistent\cite{AS}. To address this issue, we constrain the class of $\mathcal{F}$ and use sieve method to prove normality of expectile neural networks.
We refer reader to Chen for more details in the method of sieves \cite{LSSE}. Since we use asymmetic loss function, we establish the upper bounds for the empirical risk and the sample complexity based on the covering number and the Vapnik-Chervonenkis dimension \cite{AB}. The estimator of expectile neural networks can also be regarded as M-estimator\cite{EPIM}. 

The paper is organized as follows. We briefly introduce expectile neural networks in section 2. Section 3 shows the uniform law of large numbers of expectile neural networks.
In section 4, we prove the normality of expectile neural networks.

\section{Expectile neural networks}
We will briefly introduce expectile neural networks\cite{JHL}. Expectile neural network(ENN) uses asymmetric $L_2$ loss function, where we don't assume a particular functional form of covariates and use neural networks to approximate the underlying expectile regression function. By setting different $\tau$, we could get different conditional probability. By integrating the idea of neural networks into expectile regression, we propose an ENN method. We illustrate ENN with one hidden layer. The method can be easily extended to an expectile regression deep neural network with multiple layers.

\begin{figure}[h]
\begin{center}
\includegraphics[scale=0.6]{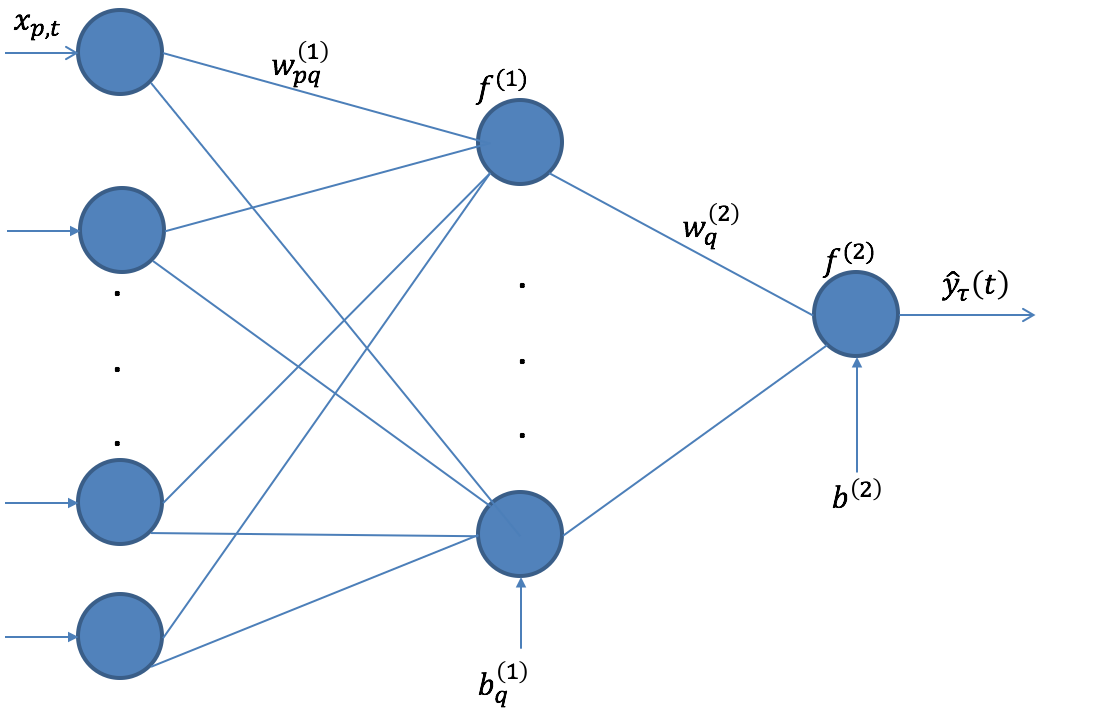}
\caption{A graphical representation of expectile neural network}
\end{center}
\end{figure}

Given the $\mathbf{x}_t$, we first build the hidden nodes $h_{q,t}$,
\begin{equation}\label{hiddenlayer}
h_{q,t}= f^{(1)}(\sum_{p=1}^{P} x_{p,t}w_{pq}^{(1)}+b_{q}^{(1)}),  q= 1,...,Q, t = 1, ..., n.
\end{equation}

where $w_{pq}$ denotes weights and $b_{q}$ denotes the bias;  $f^{(1)}$ is the activation function for the hidden layer that can be a sigmoid function, a hyperbolic tangent function, or a rectified linear units(ReLU) function. Similar to hidden nodes in neural networks, the hidden nodes in ENN can learn complex features from covariates $\mathbf{x}$, which makes ENN capable of modelling non-linear and non-additive effects. Based on these hidden nodes, we can model the conditional $\tau$-expectile, $\hat{y}_{\tau}(t)$,
\begin{equation}\label{output}
\hat{y}_{\tau}(t)=f^{(2)}(\sum_{q=1}^{Q} h_{q,t}w_{q}^{(2)}+b^{(2)}),
\end{equation}
where $f^{(2)}$, $w_{q}^{(2)}$, and $b^{(2)}$ are the activation function, weights, and bias in the output layer, respectively.$f^{(2)}$ can be identity function, sigmoid function, or a rectified linear units(ReLU) function.  A graphical representation of ENN is given in Figure 1.

From equations (\ref{hiddenlayer}) and (\ref{output}), we can have the overall function $f$:
\begin{equation}
f=f^{(2)}(\sum_{q=1}^{Q} f^{(1)}(\sum_{p=1}^{P} x_{p,t}w_{pq}^{(1)}+b_{q}^{(1)})w_{q}^{(2)}+b^{(2)}). 
\end{equation}
Then $\hat{y}_{\tau}(t) = f(\mathbf{x_i}).$
To estimate $w_{pq}^{(1)}, b^{(1)}_{q}, w_{q}^{(2)}, b^{(2)}$, we minimize the empirical risk function

\begin{equation}
\mathcal{R}(\tau)=\frac{1}{n}\sum_{i=1}^{n} L_{\tau}(y_i,f(\mathbf{x_i})),
\end{equation}
where
\begin{equation}
L_{\tau}(y_{i}, f(\mathbf{x_i}))=\left\{
\begin{aligned}
&(1- \tau) (y_{i} - f(\mathbf{x_i}))^2, & if \ y_{i} < f(\mathbf{x_i}) \\
&\tau (y_i - f(\mathbf{x_i})))^2, & if \ y_i \geq f(\mathbf{x_i}).
\end{aligned}
\right.
\end{equation}

\section{Uniform law of large numbers}

We need to consider the empirical risk of expectile neural networks. Set $Z = (X, Y), Z_i = (X_i, Y_i), i = 1,..., n$ 
where $g_1(x,y) = \vert y - f(x) \vert^2 \mathbbm{1}_{\lbrace y - f(x) \geq 0\rbrace}$, 
$g_2(x,y) = \vert y - f(x) \vert^2 \mathbbm{1}_{\lbrace y - f(x) < 0\rbrace}$ for $f \in \mathcal{F}_n$ ,  
$\mathcal{G}_{n,1} = \lbrace  \vert y - f(x) \vert^2 \mathbbm{1}_{\lbrace y - f(x) \geq 0\rbrace} : f \in \mathcal{F}_{r_n} \rbrace$ and
$\mathcal{G}_{n,2} = \lbrace  \vert y - f(x) \vert^2 \mathbbm{1}_{\lbrace y - f(x) < 0\rbrace} : f \in \mathcal{F}_{r_n} \rbrace$. The empirical risk is
\begin{equation*}
\begin{split}
R_2  &= \frac{1}{n}\sum_{i=1}^{n}\lbrace\tau \left( Y_i - f(X_i)\right)^2\mathbbm{1}_{\lbrace Y_i - f(X_i) \geq 0 \rbrace}+ (1-\tau) \left( Y_i - f(X_i)\right)^2\mathbbm{1}_{\lbrace Y_i - f(X_i) < 0\rbrace} - \\
 &   E \left( \tau \left( Y_i - f(X_i)\right)^2\mathbbm{1}_{\lbrace Y_i - f(X_i) \geq 0 \rbrace}+ (1-\tau) \left( Y_i - f(X_i)\right)^2\mathbbm{1}_{\lbrace Y_i - f(X_i) < 0\rbrace}  \right) \rbrace \\
  & = \frac{1}{n}\sum_{i=1}^{n} \tau [g_1(Z_i) - E(g_1(Z_i))] + (1 - \tau) [g_2(Z_i) - E(g_2(Z_i))].  
\end{split}
\end{equation*}

We focus on the sieve of neural networks with one hidden layer and sigmoid activation function. 
\begin{equation}
\begin{split}
\mathcal{F}_{r_n} &= \lbrace \alpha_0 + \sum_{j = 1}^{r_n} \alpha_j \sigma \left( \gamma_{j}^{T} \mathbf{x}  + \gamma_{0,j} \right) : \gamma_{j} \in \mathbb{R}^d, \alpha_j, \gamma_{0,j} \in \mathbb{R}, \sum_{j = 0}^{r_n} \vert \alpha_j\vert \leq V_n \\
& \text{ for some } V_n \geq 4 \text{ and } \max_{1 \leq j \leq r_n} \sum_{i=0}^d \vert \gamma_{i,j} \vert \leq M_n  \text{ for some } M_n > 0 \rbrace
\end{split}
\end{equation}
where 
$
r_n, V_n, M_n \rightarrow \infty \text{ as } n \rightarrow \infty.
$

To prove uniform law of large numbers, we need to introduce the lemma 1.
\begin{lem}
For $n \in \mathcal{N}$, let $\mathcal{G}_n$ be a set of functions $g: \mathcal{R}^d \rightarrow [0, B]$ and let $\epsilon > 0$. Then
\begin{equation*}
\mathbf{P} \left\lbrace \sup_{g \in \mathcal{G}_n} \Big | \frac{1}{n} \sum_{i=1}^{n} g(Z_{i}) - Eg(Z) \Big | > \epsilon \right\rbrace \leq 2\mathcal{N}(\epsilon/3, \mathcal{G}_n)e^{-\frac{2n\epsilon^2}{9B^2}}
\end{equation*}

\end{lem}
\begin{thm}[Uniform law of large numbers]\label{ULLN}
Suppose $Z = (X, Y), Z_i = (X_i, Y_i), i = 1,..., n$, if $\left[ r_n(d+2) + 1 \right]log\left[r_n(d+2) + 1 \right] = 0(n)$
we can get
\begin{equation}
\sup_{f \in \mathcal{F}_n} \big| R_2 \big | \rightarrow 0, n \rightarrow \infty
\end{equation}
\end{thm}

\begin{proof}
we want to show
\begin{equation}
\sup_{f \in \mathcal{F}_n} \Big|  \frac{1}{n}\sum_{i=1}^{n} \tau [g_1(Z_i) - E(g_1(Z_i))] + (1 - \tau) [g_2(Z_i) - E(g_2(Z_i))] \Big| \rightarrow 0
\end{equation}
\begin{equation}
\begin{split}
&\sup_{f \in \mathcal{F}_n} \Big|  \frac{1}{n}\sum_{i=1}^{n} \tau [g_1(Z_i) - E(g_1(Z_i))] + (1 - \tau) [g_2(Z_i) -E(g_2(Z_i))] \Big| \\
& \leq \sup_{g_1 \in \mathcal{G}_{n, 1}} \tau \Big| \frac{1}{n} \sum_{i =1}^{n} g_1(Z_i) - E(g_1(Z_i)) \Big| + 
\sup_{g_2 \in \mathcal{G}_{n, 2}} (1-\tau) \Big| \frac{1}{n} \sum_{i =1}^{n} g_2(Z_i) - E(g_2(Z_i)) \Big|
\end{split}
\end{equation}
Now we only need to consider first part
\begin{equation}\label{one_part}
\sup_{g_1 \in \mathcal{G}_{n, 1}} \tau \Big| \frac{1}{n} \sum_{i =1}^{n} g_1(Z_i) - E(g_1(Z_i)) \Big| \rightarrow 0
\end{equation}
The proof of second part is similar.

For $B> 0$, Let $G(x) = \sup_{g_1 \in \mathcal{G}_{n,1}}\vert g_{1}(x) \vert$,$\mathcal{G}_B = \lbrace g_1 \mathbbm{1}{\lbrace G<B \rbrace}: g_1 \in \mathcal{G}_{n,1} \rbrace$.

If $g \in \mathcal{G}$,
\begin{equation}
\begin{split}
& \Big| \frac{1}{n} \sum_{i =1}^{n} g_1(Z_i) - E(g_1(Z_i)) \Big| \\
& \leq \Big| \frac{1}{n} \sum_{i =1}^{n} g_1(Z_i) -g_1(Z_i)\mathbbm{1}_{ \lbrace G(Z_i)\leq B \rbrace} \Big| + \Big| \frac{1}{n} \sum_{i =1}^{n} g_1(Z_i)\mathbbm{1}_{ \lbrace G(Z_i)\leq B \rbrace} - E(g_1(Z_i))\mathbbm{1}_{\lbrace G(Z) \leq B \rbrace} \Big| \\
&+ \Big| \frac{1}{n} \sum_{i =1}^{n} E(g_1(Z_i))\mathbbm{1}_{\lbrace G(Z) \leq B \rbrace} - E(g_1(Z_i)) \Big| \\
& \leq \Big| \frac{1}{n} \sum_{i =1}^{n} g_1(Z_i)\mathbbm{1}_{ \lbrace G(Z_i)\leq B \rbrace} - E(g_1(Z_i))\mathbbm{1}_{\lbrace G(Z) \leq B \rbrace} \Big| + \frac{1}{n} \sum_{i=1}^{n}G(Z_i)\mathbbm{1}_{\lbrace G(Z_i > B) \rbrace} + E(G(Z)\mathbbm{1}_{\lbrace G(Z)>B \rbrace})
\end{split}
\end{equation}
By $E(G(Z)) < \infty$ and the strong law of large numbers, we get
$$
\frac{1}{n} \sum_{i=1}^{n}G(Z_i)\mathbbm{1}_{\lbrace G(Z_i > B) \rbrace} \rightarrow  E(G(Z)\mathbbm{1}_{\lbrace G(Z)>B \rbrace})
$$
Let $B$ goes to infinity,
$$
E(G(Z)\mathbbm{1}_{\lbrace G(Z)>B \rbrace}) \rightarrow 0
$$
Therefore, we only need to consider,
$$
\sup_{g_1 \in \mathcal{G}_{B}} \tau \Big| \frac{1}{n} \sum_{i =1}^{n} g_1(Z_i) - E(g_1(Z_i)) \Big| \rightarrow 0
$$
Recall that if $g$ is a function $g: \mathcal{R} \rightarrow [0, B]$, then by Hoeffding's inequality
\begin{equation}
\mathbf{P} \left\lbrace \Big| \frac{1}{n} \sum_{j=1}^n g(Z_{j}) - E(g(Z)) \Big | > \epsilon \right\rbrace \leq 2e^{-\frac{2n\epsilon^2}{B^2}}
\end{equation}
By lemma 1, we have
\begin{equation}
\mathbf{P} \left\lbrace \sup_{g_1 \in \mathcal{G}_{n, 1}} \Big| \frac{1}{n} \sum_{j=1}^n g(Z_{j}) - E(g(Z)) \Big | > \epsilon \right\rbrace \leq 2 \mathcal{N}(\epsilon/3, \mathcal{G}_{n,1}, \Vert \cdot \Vert_{\infty}) e^{-\frac{2n\epsilon^2}{B^2}}
\end{equation}
We use one result about the upper bound covering number from Theorem 14.5 in Anthony and Gartlett,
\begin{equation}
\mathcal{N}(\epsilon/3, \mathcal{F}_{r_n}, \Vert \cdot \Vert_{\infty}) \leq \left( \frac{12e\left[ r_n(d+2) + 1 \right](\frac{1}{4}V)^2}{\epsilon(\frac{1}{4}V-1)} \right)^{(r_n(d+2)+1)}
\end{equation}
Recall the definition of covering number, $\mathcal{N}(\epsilon/3, \mathcal{F}_{r_n}, \Vert \cdot \Vert_{\infty})$ is minimal $N \in \mathcal{N}(\epsilon/3, \mathcal{F}_{r_n}, \Vert \cdot \Vert_{\infty})$ such that there exist functions $f_1, ..., f_N$ with the property that for every $f \in \mathcal{F}_{r_n}$ there
is a $j = j(f) \in {1, . . . , N}$ such that
$$
\sup_{x}\vert f(x) - f_{j}(x)\vert < \epsilon
$$
Since $f(x), f_j(x)$ is close enough,  $y-f(x)$ and $y- f_j(x)$ are either negative or positive in the following situation. 
\begin{equation}
\begin{split}
&\sup_{x} \vert(y - f(x))^2\mathbbm{1}_{\lbrace y-f(x) \geq 0 \rbrace}  - (y - f_j(x))^2\mathbbm{1}_{\lbrace y-f_j(x) \geq 0 \rbrace} \vert \\
& \leq \sup_{x} \vert(y - f(x))^2 - (y - f_j(x))^2\vert \\
& = \sup_{x} \vert 2y(f_j -f) + (f-f_j)(f+f_j) \vert \\
& < 2(M_1+M_2) \epsilon
\end{split}
\end{equation}
Since $y \in \mathcal{G}_{B}$ and any functions in $\mathcal{F}_{r_n}$ are bounded, there exist $M_1$ and $M_2$ such that $\vert y\vert < M_1$ and $\vert f \vert < M_2$. 
Then $\mathcal{N}(\epsilon/3, \mathcal{G}_{n,1}, \Vert \cdot \Vert_{\infty})  \leq \mathcal{N}(\epsilon/3, \mathcal{F}_{r_n}, \Vert \cdot \Vert_{\infty})$

If $\left[ r_n(d+2) + 1 \right]log\left[r_n(d+2) + 1 \right] = 0(n)$, then
\begin{equation}
\sum_{n=1}^{\infty} \exp \left\lbrace \left[ r_n(d+2) + 1 \right] log \left( \frac{12e\left[ r_n(d+2) + 1 \right](\frac{1}{4}V)^2}{\epsilon(\frac{1}{4}V-1)} \right)\right\rbrace \cdot e^{-\frac{2n\epsilon^2}{B^2}} < \infty
\end{equation}
(\ref{one_part}) will follow by using Borel-Cantelli lemma.
\end{proof}

\section{consistency}
Since we have proven uniform laws of large numbers, we will use it to show the consistency of the neural networks.

\begin{thm}
Let $(\Omega, \mathcal{F}, P)$ be a complete probability space and let $(\Theta, \rho)$ be a metric space. Let $\left\lbrace \Theta_n \right\rbrace$ be a sequence of compact subsets of $\Theta$. Let $Q_n: \Omega \times \Theta_n \rightarrow \overline{\mathbb{R}}$ be measurable $\mathcal{F} \times \mathcal{B}(\Theta_n)/ \overline{\mathcal{B}}$, and suppose that for each $\omega$ in $\Omega$, $Q_n(\omega, \cdot)$ is lower semicontinuous on 
$\Theta_n, n = 1, 2, ....$

Then for each $n = 1, 2, ...$ there exists $\hat{\theta}_n: \Omega \rightarrow \Theta_n$ 
measurable $\mathcal{F}/\mathcal{B}(\Theta_n)$ such that for each $\omega$ in $\Omega$, 
$Q_{n}(\omega, \hat{\theta}_n(\omega)) = \inf_{\theta \in \Theta_n} Q_n(\omega, \Theta).$
\end{thm}
The proof of theorem 2 is omitted. Interested readers can refer to White and Wooldridge\cite{WW}

\begin{lem}\label{cpt}
Let $\chi$ be a compact subset of $\mathbb{R}^d$, then for each fixed n, $\mathcal{F}_{r_n}$ is a compact set.
\end{lem}

Suppose the true expectile neural networks is
\begin{equation}
y_i = f_{0}(\mathbf{x}_i) + \epsilon_i,
\end{equation}
where $\epsilon_i, ..., \epsilon_n$ are i.i.d. random variables defined on a complete probability space $(\Omega, \mathcal{A}, \mathbb{P})$ and $E(\epsilon_i) = 0$, $Var(\epsilon_i) = \sigma^2 < \infty$
\begin{equation}
Q_n(f)= \frac{1}{n} \sum_{i=1}^n E \left[ \left(\vert \tau - \mathbbm{1}_{\lbrace y_i < f(x_i)\rbrace} \vert (y_i - f(\mathbf{x}_i) )^2 \right) \right]
\end{equation}
We check the condition
\begin{lem}\label{CD3}
Suppose  
$$
\frac{\epsilon^2 min\left\lbrace \tau, 1- \tau \right\rbrace}{max\left\lbrace \tau, 1- \tau \right\rbrace - min\left\lbrace \tau, 1- \tau \right\rbrace} = \frac{\epsilon^2 min\left\lbrace \tau, 1- \tau \right\rbrace}{\vert 1- 2 \tau \vert} > \sigma^2, \text{if } \tau \neq \frac{1}{2},
$$ \text{then} $ \inf_{f: \Vert f - f_0 \Vert_n \geq \epsilon} Q_n(f) - Q_n(f_0) > 0$
\end{lem}
\begin{proof}
\begin{equation}
\begin{split}
Q_n(f)&= \frac{1}{n} \sum_{i=1}^n E \left[ \left(\vert \tau - \mathbbm{1}_{\lbrace y_i < f(\mathbf{x_i})\rbrace} \vert (y_i - f(\mathbf{x}_i) )^2 \right) \right]\\
&\geq \frac{1}{n} \sum_{i=1}^n E\left[  min\left\lbrace \tau, 1- \tau\right\rbrace (y_i - f(\mathbf{x}_i) )^2  \right]\\
& = \frac{1}{n} min\left\lbrace \tau, 1- \tau \right\rbrace \sum_{i=1}^n \left[ \left( f(\mathbf{x}_i) - f_0(\mathbf{x}_i)\right)^2 + \sigma^2 \right]
\end{split}
\end{equation}

\begin{equation}
\begin{split}
Q_n(f_0)& = \frac{1}{n} \sum_{i=1}^n E \left[ \left(\vert \tau - \mathbbm{1}_{\lbrace y_i < f_{0}(\mathbf{x_i})\rbrace} \vert (y_i - f_0(\mathbf{x}_i) )^2 \right) \right]\\
&  \leq max\left\lbrace \tau, 1- \tau \right\rbrace \sigma^2.
\end{split}
\end{equation}

\begin{equation}
\begin{split}
&\inf_{f: \Vert f - f_0 \Vert_n \geq \epsilon} Q_n(f) - Q_n(f_0) \\
&\geq \inf_{f: \Vert f - f_0 \Vert_n \geq \epsilon} \frac{1}{n} min\left\lbrace \tau, 1- \tau \right\rbrace \sum_{i=1}^n \left[ \left( f(\mathbf{x}_i) - f_0(\mathbf{x}_i)\right)^2 + \sigma^2 \right] - max\left\lbrace \tau, 1- \tau \right\rbrace \sigma^2.\\
& = min\left\lbrace \tau, 1- \tau \right\rbrace(\sigma^2 + \epsilon^2) - max\left\lbrace \tau, 1- \tau \right\rbrace \sigma^2\\
& > 0
\end{split}
\end{equation}
\end{proof}

\begin{thm}
Under the notation given above, if
$$\frac{\epsilon^2 min\left\lbrace \tau, 1- \tau \right\rbrace}{max\left\lbrace \tau, 1- \tau \right\rbrace - min\left\lbrace \tau, 1- \tau \right\rbrace} > \sigma^2, \text{if } \tau \neq \frac{1}{2},$$
then $\Vert \hat{f}_n - f_0 \Vert_n \overset{p}{\to} 0$
\end{thm}
\begin{proof}
By corollary 2.6 in White and Wooldridge(1991) with Theorem \ref{ULLN}, lemma \ref{CD3}, lemma \ref{cpt} and, we have
$$
\Vert \hat{f}_n - f_0 \Vert_n \overset{p}{\to} 0
$$
\end{proof}

\section{Normality}
We will use the following theorem to prove the normality of expectile neural network [3].
\begin{thm}
Suppose that $\mathcal{F}$ is a $P-$Donsker class of measurable functions and $\hat{f_n}$ is a
sequence of random functions that take their values in $\mathcal{F}$ such that
$$
\int \left( \hat{f_n}(x) - f_{0}(x) \right) ^2 dP(x) \overset{P}{\to} 0
$$
for some $f_0 \in L_{2}(P)$. Then
$$
\frac{1}{\sqrt{n}} \sum_{i=1}^{n} \left( (\hat{f_n} - f_0)(X_i) - P(\hat{f_n} -f_0) \right) \overset{P}{\to} 0,
$$
and
$$
\frac{1}{\sqrt{n}} \sum_{i=1}^{n} \hat{f_n}(X_i) -P\hat{f_n} \sim N(0, Pf_{0}^2 -(Pf_0)^2)
$$
\end{thm}

From theorem 4, We need to check two conditions
\begin{itemize}
\item $\mathcal{F}_{r_n}$ is $P-$Donsker class
\item$\int \left( \hat{f_n}(x) - f_{0}(x) \right) ^2 dP(x) \overset{P}{\to} 0$
\end{itemize}
The proof of $\mathcal{F}_{r_n}$ is Donsker class can be found in Van der Vaart and A.W., Wellner \cite{WCEP} . we use theorem 5 to check $\int \left( \hat{f_n}(x) - f_{0}(x) \right) ^2 dP(x) \overset{P}{\to} 0.$
\begin{thm}
Let $\sigma$ be a squashing function. Then, for every probability measure $\mu$ on $\mathcal{R}^d$, every measurable $f: \mathcal{R}^d \rightarrow \mathcal{R}$ with $\int \vert f(x) \vert^{2} \mu(dx) < \infty$, and every $\epsilon > 0$, there exists a neural network $h(x)$ in 
$$
h(x)= \lbrace \sum_{i=1}^{k}c_i \sigma(a_i^{T}x + b_i) + c_0: k \in \mathbf{N}, a_i \in \mathcal{R}^d, b_i, c_i \in \mathcal{R} \rbrace
$$
such that
$$
\int \vert f(x) - h(x) \vert^{2} \mu(dx) < \epsilon 
$$
\end{thm}

Next, we will establish the asymptotic normality of ENN.
We assume that $f_{0} \in \mathcal{F}$, where $\mathcal{F}$ is the class of continuous functions with compact supports. $f_0$ is a function needed to be estimated.

\begin{thm}
Suppose $\hat{f}_{n}(x) \in \mathcal{F}$ is a sequence of random functions and $\int \vert f_{0}(x) \vert^{2} dP(x) < \infty $. if conditions in consistency exist, 
we can get
$$
\int \left( \hat{f_n}(x) - f_{0}(x) \right) ^2 dP(x) \overset{P}{\to} 0
$$
for some $f_0 \in L_{2}(P)$. Then
$$
\frac{1}{\sqrt{n}} \sum_{i=1}^{n} \left( (\hat{f_n} - f_0)(X_i) - P(\hat{f_n} -f_0) \right) \overset{P}{\to} 0,
$$
and
$$
\frac{1}{\sqrt{n}} \sum_{i=1}^{n} \hat{f_n}(X_i) -P\hat{f_n} \sim N(0, Pf_{0}^2 -(Pf_0)^2)
$$
\end{thm}
\begin{proof}
Let $\pi_{r_n}f_{0} \in \mathcal{F}_{r_n}$
\begin{equation}
\begin{split}
&\Vert \hat{f}_{n}(x) - f_{0}(x) \Vert^{2} \leq \Vert \hat{f}_{n} - \pi_{n}f_{0}  \Vert^{2} + \Vert \pi_{n}f_0 -f_0 \Vert^{2} \\
\end{split}
\end{equation}
Using the result of proving consistency of ENN
$$\Vert \hat{f}_{n} - \pi_{n}f_{0}  \Vert^{2} \overset{p}{\to} 0$$
By theorem 5,
$$\Vert \pi_{n}f_0 -f_0 \Vert^{2} < \epsilon$$ 
With theorem 6, we can get the result.
\end{proof}


\section*{Acknowledgment}
This work was supported by NIH 1R01DA043501-01 and NIH 1R01LM012848-01.

\end{document}